\documentclass[12pt]{article}
\usepackage{amssymb,amsthm,amsmath,amsfonts,latexsym,tikz,hyperref}
\usepackage[hmargin=1in,vmargin=1in]{geometry}

\newtheorem{thm}{Theorem}%[section]
\newtheorem{prop}[thm]{Proposition}
\newtheorem{cor}[thm]{Corollary}

\newcommand{\fS}{{\mathfrak S}}
\newcommand{\cT}{{\cal T}}
\newcommand{\SSYT}{SSYT}
\newcommand{\cI}{{\cal I}}
\DeclareMathOperator{\des}{des}
\DeclareMathOperator{\Des}{Des}
\DeclareMathOperator{\sh}{sh}
\DeclareMathOperator{\st}{st}

\DeclareMathOperator{\sign}{sign}
\DeclareMathOperator{\Asc}{Asc}

\definecolor{vividviolet}{rgb}{0.62, 0.0, 1.0}

\usepackage{nicefrac}
\usepackage{xfrac}
\usepackage{youngtab} % for Young tableaux
\usepackage{young} % for Young tableaux

\begin{document}
\pagestyle{plain}

\title{Combinatorial Proofs of Identities Involving Symmetric Matrices}
\author{Samantha Dahlberg\\[-5pt]
\small Department of Mathematics, Michigan State University,\\[-5pt]
\small East Lansing, MI 48824-1027, USA, {\tt dahlbe14@msu.edu}}

\date{\today\\[10pt]
	\begin{flushleft}
	                                       \\[5pt]
	\end{flushleft}}

\maketitle

\begin{abstract}

Brualdi and Ma found a connection between involutions of  length $n$ with $k$ descents and symmetric $k\times k$ matrices with non-negative integer entries summing to $n$ and having no row or column of zeros. From their main theorem they derived two alternating sums by algebraic means and asked for combinatorial proofs.  In this note we provide such demonstrations making use of the Robinson-Schensted-Knuth correspondence between symmetric matrices and semi-standard Young Tableau. Additionally, we restate the proof of Brualdi and Ma's main result with this perspective which shortens the argument. 

%S: Abstracts are to be stand-alone and not cite references. Also it's important to contrast their algebraic proof with your combinatorial one. }
\end{abstract}

\begin{section}{Introduction}

Let $\fS_n$  be the collection of permutations on $[n]=\{1,2,\dots,n\}$ and $\cI_n=\{\pi\in\fS_n:\pi^2=\text{id}\}$ be the collection of involutions. A permutation $\pi \in \fS_n$ has a {\it descent} at index $i$ if $\pi(i)>\pi(i+1)$. We write $\Des(\pi)=\{i:\pi(i)>\pi(i+1)\}$ for the collection of {\it descents} in $\pi$ and $\Asc(\pi)=[n-1]\setminus \Des(\pi)$ for the set of  {\it ascents}. The {\it number of descents} is $\des(\pi)=|\Des(\pi)|$
and  the generating function for the descent statistic on involutions is 
 $$I_n(t)=\sum_{\iota\in\cI_n}t^{\des(\iota)}=\sum_{k=0}^{n-1}I(n,k)t^k$$
where we denote the set of involutions $\iota \in \cI_n$ with $\des(\iota)=k$ as $\cI(n,k)$ and the number of such involutions $I(n,k)=|\cI(n,k)|$. This generating function has been studied by many  and is known to be symmetric~\cite{S81} and unimodal~\cite{GZ06}~\cite{D07} but is not log-concave~\cite{BBS09}. 
Additionally, Brualdi and Ma~\cite{BM15}, whose work this paper is in response to, studied this function and  found a connection to symmetric matrices.
Let $\cT(n,k)$ be the collection of symmetric $k\times k$ matrices with non-negative integer entries which satisfy the following two conditions. 
\begin{enumerate}
\item  The entries sum up to $n$. 
\item No row or column contains only zero entries. 
\end{enumerate}
Let $T(n,k)=|\cT(n,k)|$. 
%Since the only matrices in this note are square we will define the {\it dimension} of a $k\times k$ matrix $X$ to be $\dim(X)=k$.
The main result of Brualdi and Ma is below.

\begin{thm}[Brualdi and Ma~\cite{BM15} Theorem 1] For $n\geq 1$ we have
$$I_n(t)=\sum_{k=1}^nT(n,k)t^{k-1}(1-t)^{n-k}$$ 
which is equivalent to 
\begin{equation}\sum_{k=0}^{n-1}I(n,k)t^{k+1}(1+t)^{n-k-1}=\sum_{k=1}^nT(n,k)t^k.
\label{eq:BM}
\end{equation}
\label{MainThm} \end{thm}

Cho in~\cite{C08} found that $I_n(t)$ could be expressed using Kotska numbers, whose definition we build next. 
An integer partition $\lambda=(\lambda_1,\lambda_2,\dots,\lambda_{l})$ of $n$, $\lambda\vdash n$, is a weakly decreasing sequence of positive integers which sum to $n$. 
%,  $\sum \lambda_i=n$. 
The {\it Ferrer's diagram} of an integer partition $\lambda$ is a depiction of $l$ left-aligned rows of boxes where  row $i$ has $\lambda_i$ boxes and row one is at the top. If we fill each box with  positive integers such that each row is weakly increasing and each column is strictly increasing then we call the diagram a {\it semi-standard Young Tableau}, SSYT. The {\it shape} of a SSYT, $\sh(T)$, is the underlying integer partition or Ferrer's diagram. The {\it content}, $C(T)$, of a SSYT refers to the collection of fillings usually written as an {\it integer composition},  a sequence of positive integers $C=(a_1,a_2,\dots,a_k)$ which sum to $n$ with {\it length} $\ell(C)=k$. If the content of a SSYT is $C(T)=(a_1,a_2,\dots,a_k)$ then  $T$ is filled with $a_1$ ones, $a_2$ twos and $a_i$ $i$'s. In general the content of a SSYT is actually a {\it weak integer composition} which allows for zero entries, that $a_i=0$, since it is not a requirement for a SSYT to have at least one filling of $1,2,\dots, n$ for some $n$. Given a SSYT we can {\it standardize} $T$, $\st(T)$, by replacing all the $i$th smallest fillings with $i$. If $\st(T)=T$ we will call this SSYT {\it proper} which we specifically  focus on  in this paper. We call a SSYT a {\it standard Young Tableau}, SYT, if $T$ is filled with $1,2,\dots, n$. The {\it Kotska numbers}, $K_{\lambda,C}$, count the number of proper SSYT of shape $\lambda$ and content $C$. 
Cho~\cite{C08} found that 
$$I_n(t)=\sum_{k = 1}^{n}\sum_{\lambda\vdash n, \ell(C)=k}K_{\lambda,C}t^{k-1}(1-t)^{n-k}$$
illustrating a connection between involutions and SSYT.

One corollary which Brualdi and Ma concluded using this equation was
\begin{equation}
T(n,k)=\sum_{\lambda\vdash n, \ell(C)=k}K_{\lambda,C}
\label{eq:Kotska}
\end{equation}
which was originally found by Cho in~\cite{C08} which does suggest that  there is a bijection from symmetric matrices and SSYT. This is actually the well-known Robinson-Schensted-Knuth, RSK, correspondence. This correspondence is often introduced as a bijection between permutations $\pi$ and  pairs of SYT   $(P,Q)$ of the same shape. A comprehensive introduction of  RSK in this light can be found in~\cite{S01} where the original presentations were done by Robinson~\cite{R38} and Schensted~\cite{S61}. This correspondence has the nice property that the inverse $\pi^{-1}$ corresponds to the reverse pair $(Q,P)$. This in turn implies that an involution corresponds to a singe SYT. Knuth~\cite{K70} generalized this correspondence between matrices $M$ and pairs of SSYT $(P,Q)$ of the same shape. This generalizes the previous correspondence if you consider a permutation as its permutation matrix. This generalization has the nice property that the  transpose of the matrix,
% $M^{\perp}$, 
 $M^{T}$, 
corresponds to the reversed pair $(Q,P)$ which implies the RSK correspondence ties together a symmetric matrix and a single SSYT. Particularly, the symmetric matrices dealt with by Brualdi and Ma pair with proper SSYT as the non-zero row  and column condition gives this extra restriction on the fillings. Brualdi and Ma had two corollaries which they proved algebraically from Theorem~\ref{eq:BM}. 

\begin{cor}[Brualdi and Ma~\cite{BM15} Corollary 3]
We have for $n\geq 1$ that $$\sum_{k=1}^n (-1)^kT(n,k)=(-1)^n.$$
\label{cor1}
\end{cor}

Brualdi and Ma also showed that equation~\eqref{eq:BM} still holds when restricting to symmetric matrices in $\cT(n,k)$ with zeros along its main diagonal and involutions without fixed points, meaning they do not have any one-cycles. Let  $W(n,k)$ count the number of matrices in $\cT(n,k)$ such that the main diagonal is all zeros.  Brualdi and Ma's second corollary which we consider in this paper is the following. 
\begin{cor}[Brualdi and Ma~\cite{BM15} Corollary 5]We have for even $n$ that $$\sum_{k=2}^n(-1)^kW(n,k)=1.$$
\label{cor2}
\end{cor}

They ask in their paper for combinatorial proofs to these corollaries and here we response to that. Our combinatorial proofs use a sign-reversing involution and rely on the   RSK correspondence which we also use to  re-present Brualdi and Ma's bijective proof of Theorem~\ref{MainThm} as this view shortens the argument. 
Since we heavily use the correspondence between symmetric matrices and SSYT  we describe in the next section a quick method to determine corresponding objects by Burge~\cite{B74} which is restated here. In the last section we present the combinatorial proofs of the corollaries  and end the section with a re-presentation of Brualdi and Ma's proof of their main theorem restated in terms of SSYT rather than symmetric matrices. 

\end{section}
\begin{section}{RSK between symmetric matrices and  SSYT}

Knuth in his  paper~\cite{K70}  describes the correspondence between any $n\times m$ matrix with non-negative integer entries to a pair of SSYT $(P,Q)$ of the same shape where the sequence of column sums is the content of $P$ and the sequence of row sums is the content of $Q$. We give an example in Figure~\ref{fig:RSK}. He particularly shows that if $M$ corresponds to the pair of tableau $(P,Q)$ then the transpose $M^T$ corresponds to the pair of tableau $(Q,P)$, so a symmetric matrix $M$ corresponds to a single  SSYT $T$. This correspondence  has the property that if a symmetric matrix $M$ corresponds to $T$ and $M$ has $r_i$ as the sum of entries in the  $i$th row then the content of $T$ is  $({r_1},r_2,\dots {r_k})$ which can be seen in the example in Figure~\ref{fig:MtoSSYT}. Note that this means if $M$ has no  row or columns of all zeros then $M$ corresponds to a proper SSYT. 

\begin{figure}[h]
\begin{center}
\begin{tikzpicture}[scale = .5]
\begin{scope}[shift={(0,0)}]
%\draw (-3.3,0) node {$M=$};
\draw (0,0) node {$\left[
\begin{array}{ccc}
0&0&3\\
2&1&1\\
0&1&0
\end{array}\right]$};
\draw (3,0) node {$\leftrightarrow$};
\end{scope}
\begin{scope}[shift={(6.5,0)}]
%\draw (-3.3,0) node {$T=$};
\draw (0,0) node { \young(1122,3333) };
\draw (2.5,-1) node { ,};
\draw (5,0) node { \young(1112,2223) };
\draw (-2,2) arc (150:210:3.5) ;
\draw (7,-1.5) arc (-30:30:3.5) ;
 \end{scope}
 \end{tikzpicture}
%\begin{tikzpicture}[scale = .5]
%\draw (-2,3) node {$M=$};
%\draw (0,-0)--(-.5,0)--(-.5,6.5)--(0,6.5);
%\draw (5.5,0)--(6,0)--(6,6.5)--(5.5,6.5);
%\draw (0,.2)--(1.5,.2)--(1.5,1)--(4.5,1)--(4.5,4.2)--(5.6,4.2)--(5.6,6)--(0,6)--(0,.2);
%\draw (2,.5) node {$x$};
%\draw (2.9,.5) node {$0$};
%\draw (4.1,.5) node {$\cdots$};
%\draw (5.2,.5) node {$0$};
%\draw (5.2,2) node {$\vdots$};
%\draw (5.2,2.9) node {$0$};
%\draw (5.2,3.7) node {$x$};
% \end{tikzpicture}
%\begin{tikzpicture}[scale = .5]
%\draw (-2,3) node {$M'=$};
%\draw (0,-0)--(-.5,0)--(-.5,6.5)--(0,6.5);
%\draw (5.5,0)--(6,0)--(6,6.5)--(5.5,6.5);
%\draw (0,.2)--(1.5,.2)--(1.5,1)--(4.5,1)--(4.5,4.2)--(5.6,4.2)--(5.6,6)--(0,6)--(0,.2);
%\draw (2,.5) node {$x-1$};
%\draw (2.9,.5) node {$0$};
%\draw (4.1,.5) node {$\cdots$};
%\draw (5.2,.5) node {$0$};
%\draw (5.2,2) node {$\vdots$};
%\draw (5.2,2.9) node {$0$};
%\draw (5.2,3.7) node {$x-1$};
% \end{tikzpicture}
\end{center}
\caption{Matrix $M$ on the left corresponds by RSK to the pair of SSYT $(P,Q)$ on the right with $C(P)=(2,2,4)$ and $C(Q)=(3,4,1)$. }
\label{fig:RSK}
\end{figure}

Since we rely on this correspondence between symmetric matrices and SSYT we present in this section a quick algorithm originally introduced by Burge in~\cite{B74} which bypasses the need to find the word associated to the matrix and the full bumping algorithm. This algorithm is the generalization of Beissinger~\cite{B87} who describes the further tie between involutions and their associated SYT. Her algorithm  considers the cycle notation of an involution and describes the associated SYT inductively  as you include more two-cycles.  Specifically, given an involution $\iota'\in \cI_{n-2}$ we can add another two-cycle $(i,n)$ with $1\leq i < n$ by doing the following in one-line notation. First increase all numbers greater than or equal to $i$ by one and then have this permutation written in the indices $[n]\setminus \{i,n\}$. Finally we get our new involution by placing $i$ at index $n$ and  $n$  at index $i$.  Beissinger notates this new involution $\iota=\iota'+(i,n)$. For example $45312+(3,6)=5674123$. Her  algorithm determines the SYT of $\iota$ from the SYT $T'$ of $\iota'$ with the following three steps. 
\begin{enumerate}
\item Increase all fillings greater than or equal to $i$ in $T'$ by one. 
\item Insert $i$ into $T'$ as the Robinson-Schensted-Knuth bumping algorithm dictates. The bumping path will end on some row $r$. 
\item Insert $n$ at the end of row $r+1$. 
\end{enumerate}
With these three steps we arrive at the SYT $T$ for $\iota=\iota'+(i,n)$.

\begin{figure}[h]
\begin{center}
\begin{tikzpicture}[scale = .5]
\begin{scope}[shift={(0,0)}]
\draw (-3.3,0) node {$M=$};
\draw (0,0) node {$\left[
\begin{array}{ccc}
0&1&2\\
1&3&0\\
2&0&0
\end{array}\right]$};
\draw (3,0) node {$\leftrightarrow$};
\end{scope}
\begin{scope}[shift={(8,0)}]
\draw (-3.3,0) node {$T=$};
\draw (0,0) node { \young(1112,222,33) };
 \end{scope}
 \begin{scope}[shift={(16,0)}]
\draw (-3.3,0) node {$M'=$};
\draw (0,0) node {$\left[
\begin{array}{ccc}
0&1&1\\
1&3&0\\
1&0&0
\end{array}\right]$};
\draw (3,0) node {$\leftrightarrow$};
\end{scope}
\begin{scope}[shift={(24,0)}]
\draw (-3.3,0) node {$T'=$};
\draw (0,0) node { \young(1122,22,3) };
 \end{scope}
 \end{tikzpicture}
%\begin{tikzpicture}[scale = .5]
%\draw (-2,3) node {$M=$};
%\draw (0,-0)--(-.5,0)--(-.5,6.5)--(0,6.5);
%\draw (5.5,0)--(6,0)--(6,6.5)--(5.5,6.5);
%\draw (0,.2)--(1.5,.2)--(1.5,1)--(4.5,1)--(4.5,4.2)--(5.6,4.2)--(5.6,6)--(0,6)--(0,.2);
%\draw (2,.5) node {$x$};
%\draw (2.9,.5) node {$0$};
%\draw (4.1,.5) node {$\cdots$};
%\draw (5.2,.5) node {$0$};
%\draw (5.2,2) node {$\vdots$};
%\draw (5.2,2.9) node {$0$};
%\draw (5.2,3.7) node {$x$};
% \end{tikzpicture}
%\begin{tikzpicture}[scale = .5]
%\draw (-2,3) node {$M'=$};
%\draw (0,-0)--(-.5,0)--(-.5,6.5)--(0,6.5);
%\draw (5.5,0)--(6,0)--(6,6.5)--(5.5,6.5);
%\draw (0,.2)--(1.5,.2)--(1.5,1)--(4.5,1)--(4.5,4.2)--(5.6,4.2)--(5.6,6)--(0,6)--(0,.2);
%\draw (2,.5) node {$x-1$};
%\draw (2.9,.5) node {$0$};
%\draw (4.1,.5) node {$\cdots$};
%\draw (5.2,.5) node {$0$};
%\draw (5.2,2) node {$\vdots$};
%\draw (5.2,2.9) node {$0$};
%\draw (5.2,3.7) node {$x-1$};
% \end{tikzpicture}
\end{center}
\caption{ The pairs $M$ and $T$  correspond by RSK as well as $M'$ and $T'$.}
\label{fig:MtoSSYT}
\end{figure}

The short-cut RSK algorithm for matrices in $\cT(n,k)$ and their SSYT by Burge~\cite{B74} is very similar. Though we aren't adding on two-cycles anymore like we do for involutions the set-up is similar to Beissinger's. Given a matrix $M\in \cT(n,k)$ we find the largest column $a$ such that the $k$th row is all zeros for all larger columns, $M_{k,j}=0$ for all $j>a$, which we illustrate in Figure~\ref{fig:MtoSSYT}. Consider the matrix $M'$ which is matrix $M$ with the entries $M_{k,a}$ and $M_{a,k}$ decreased by one. Let $T'$  be the  SSYT associated to $M'$.  There are some cases where this set up or the following steps do not make sense, and we will mention these cases after we describe the general algorithm. We get the  SSYT $T$ associated to $M$ in the following two steps.
\begin{enumerate}
\item Insert $a$ into $T'$ as the RSK bumping algorithm dictates. The bumping path will end on some row $r$. 
\item Insert $k$ at the end of row $r+1$. 
\end{enumerate}
The first deviant case is when $a = k$. In this case we decrease this single entry in $M$ by one to get $M'$ which is associated to some SSYT $T'$. We get $T$ by placing a $k$ at the end of the first row in $T'$. The next and last deviant case is when reducing the entries in $M_{k,a}$ and $M_{a,k}$ by one makes a row and column of all zeros. In this case $M'$ will be the matrix we get by removing the $a$th row and column, and our algorithm then will include step 1 from Beissinger's  algorithm in that we will first increase all fillings greater than or equal to $a$ in $T'$ by one. We can see that this does generalize  Beissinger's algorithm on involutions when you consider involutions in their matrix form. 

Since we will rely on this correspondence in the rest of the paper we formally state Knuth's result here. 
\begin{prop}[Knuth~\cite{K70} Theorem 4]
There is a correspondence between  matrices in SSYT in ${\cal T}(n,k)$ to SSYT in the set $\{\text{proper }\SSYT \text{ }T:\sh(T)\vdash n \text{ and } \ell(C(T))=k\}$. 
\end{prop}

\end{section}

\begin{section}{Combinatorial Proofs}

In this section we prove the two identities in Corollaries~\ref{cor1} and~\ref{cor2} due to Brualdi and Ma. In their paper they presented algebraic proofs and asked for combinatorial ones which we provide here. At the end of the section we re-present their proof for Theorem~\ref{MainThm} but instead use the correspondence to SSYT which shorten the proof. 

We will prove the identities in Corollaries~\ref{cor1} and~\ref{cor2}  by defining a sign-reversing involution. In order to define a sign-reversing involution we will first need to define a  {\it signed set}  $S$ which is a set  such that each $s\in S$ has an associated  {\it sign}, $\sign(s)$, of $+1$ or $-1$. The elements assigned the value $+1$ are called the positive elements and the elements assigned $-1$ are the negative elements. A {\it sign-reversing involution} is a map $f:S\rightarrow S$ which is an involution,  $f\circ f = id$. The map $f$ also has to be sign reversing in that if $f(s)\neq s$ for $s\in S$  the sign associated to $s$ is opposite of the sign associated to $f(s)$. The consequence of such a map is a pairing between many elements in $S$ such that each pair has a positive element and a negative element so that when we add the  signs of the pair together we get zero. Not all elements $s\in S$ will be part of a pairing. This happens when $f(s)=s$ and we call these elements {\it fixed points}. Regarding summations this involution gives us
$$\sum_{s\in S}\sign(s)=\sum_{\text{fixed points $s$}}\sign(s)$$
which simplifies the summation. 

In our proof we will not directly define a sign-reversing involution on symmetric matrices. Instead we will use the correspondence between symmetric matrices and SSYT meaning we will define a sign-reversing involution on SSYT. We will particularly be relying on equation~\eqref{eq:Kotska} and will prove the  identity
$$\sum_{k=1}^n \sum_{\lambda \vdash n, \ell(C)=k}(-1)^kK_{\lambda,C}=(-1)^n$$
which is equivalent to Corollary~\ref{cor1}. 

%\begin{thm}[Brualdi and Ma~\cite{BM15} \tvi{say thm number}]
%We have for $n\geq 1$ that $$\sum_{k=1}^n (-1)^kT(n,k)=(-1)^n.$$
%\label{thm1}
%\end{thm}

\begin{proof}[Proof of Corollary~\ref{cor1}]
Let our signed set $S$  contain all proper SSYT with $\sh(T)\vdash n$. The sign of each $T\in S$ will be $\sign(T)=(-1)^{\ell(C(T))}$. Now we have 
$$\sum_{k=1}^n \sum_{\lambda \vdash n, \ell(C)=k}(-1)^kK_{\lambda,C}=\sum_{T\in S}\sign(T).$$
Next we define our involution. Given a $T\in S$ look at its first column. There will be some largest $a$ in this column such that all $i\in [a-1]$ appear above $a$ in column one and the rest of the tableau does not contain any $i\in [a-1]$. Fixing an $a$ we can split all such tableau  into two cases. 
\begin{enumerate}
\item The rest of the tableau contains another $a$. 
\item The rest of the tableau does not contain another $a$. 
\end{enumerate}
Now we will define the involution. Let $P$ indicate all the cells and fillings  outside the identified $1,2,\dots,a$ in the first column. If you are in the first case then the rest of the tableau, $P$, contains another $a$. We will increase all fillings in $P$ by one which gives us $P+1$. Note that because there was this additional $a$ in $P$ we still have a proper SSYT and $P+1$  also doesn't contain any $i\in [a-1]$. Further this resulting SSYT is now part of case 2 since $P+1$   doesn't contain an $a$.  If we  instead are in the second case then $P$ doesn't contains an $a$. Note that  the cell just below $a$ in the first column has to be  strictly larger than $a+1$ otherwise we contradict our choice of a largest $a$. We will decrease all fillings in cells of $P$ by one which gives us $P-1$ and by our note assures us we have a proper SSYT. Since  $P-1$ still doesn't have any $i\in [a-1]$ and by the fact that there must have been an $a+1$ in $P$ we then have an $a$ in $P-1$ which brings us  back to case 1. See Figure~\ref{fig:signinvo} for an example.

\begin{figure}[h]
\begin{center}
\begin{tikzpicture}[scale = .5]
\begin{scope}[shift={(0,0)}]
\draw[-] (-.5,2)--(3,2)--(3,0)--(2,0)--(2,-2)--(1,-2)--(1,-4)--(-.5,-4)--(-.5,2);
\draw (0,0) node { \young(1,2,\vdots,a) };
\draw (1.5,1) node {$P$};
\draw (4,0) node {$\leftrightarrow$};
\end{scope}
\begin{scope}[shift={(6,0)}]
\draw[-] (-.5,2)--(3,2)--(3,0)--(2,0)--(2,-2)--(1,-2)--(1,-4)--(-.5,-4)--(-.5,2);
\draw (0,0) node { \young(1,2,\vdots,a) };
\draw (1.7,1) node {$P-1$};
\end{scope}
\begin{scope}[shift={(15,0)}]
\draw (0,0) node { \young(144,25,36,5) };
\draw (2,0) node {$\leftrightarrow$};
\draw (4.2,0) node { \young(133,24,35,4) };
\end{scope}
 \end{tikzpicture}
\end{center}
\caption{On the left the fillings of $P$ do not contain an $a$ and $P-1$ does contain an $a$.  On the right we have a specific example of the sign-reversing involution with $a = 3$. }
\label{fig:signinvo}
\end{figure}
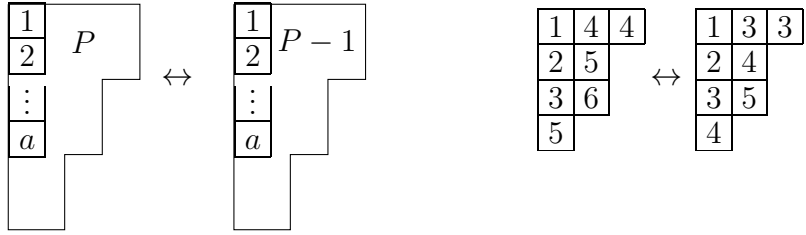

This is certainly an involution on all proper SSYT which have $P\neq \emptyset$ and is certainly sign reversing since we increase or decrease the length of the content by one.  If $P= \emptyset$ then  $T$ must be a single column filled with $1,2,\dots ,n$ which we will call $T_n$. We will let $T_n$ be a fixed point. It follows that 
$$\sum_{T\in S}\sign(T)=\sign(T_n)=(-1)^n$$
which completes the proof.
\end{proof}

Brualdi and Ma also concluded Corollary~\ref{cor2}. Recall that  $W(n,k)$ counts the number of matrices in $\cT(n,k)$ such that the main diagonal is all zeros. To prove Corollary~\ref{cor2} all we need to do is to assure that our sign-reversing involution in the proof of Corollary~\ref{cor1} does restrict to SSYT associated to matrices with zeros along its diagonal. We  get this quickly by the following result of Knuth.

\begin{prop}[Knuth~\cite{K70} Theorem 4]
Given a SSYT $T$ and its associated symmetric matrix $M$ the number of odd columns in $T$ equals the trace of $M$. 
\label{Knuth2}
\end{prop}

Note that our sign-reversing involution does not change the shape of the tableau. This means that if $T$ doesn't have an odd length column then $T$ is mapped to another tableau with the same shape which then also doesn't have an odd length column. Thus, our sign-reversing involution does indeed restrict  to SSYT associated to symmetric matrices with zeros along the main diagonal. 

%\begin{cor}[\cite{BM15}]We have for even $n$ that $$\sum_{k=2}^n(-1)^kW(n,k)=1.$$
%\end{cor}
\begin{proof}[Proof of Corollary~\ref{cor2}]By Proposition~\ref{Knuth2} our sign-reversing involution in the proof of Corollary~\ref{cor1} restricts to the collection of SSYT associated to symmetric matrices counted by $\sum_kW(n,k)$. The only fixed point here is again the tableau with a single column filled with  $1,2,\dots, n$ which has a positive sign since $n$ is even.  
\end{proof}

Finally, we will re-prove Brualdi's and Ma's Theorem~\ref{MainThm} by proving equation~\eqref{eq:BM}. The heart of Brualdi's and Ma's proof was  showing the identity
$$T(n,i)=\sum_{k = 0}^{n-i}I(n,k)\binom{n-1-k}{n-i}.$$
They proved this by defining a bijection from matrices in $\cT(n,i)$ to an involution $\iota$ 
%with $k$ descents 
paired with a set  $A\subseteq \Asc(\iota)$ of size $n-i$. We present the proof here in a new light using the correspondence between symmetric matrices and SSYT. 
We will describe a bijection from proper SSYT of shape $\lambda\vdash n$ with content length $i$ to a pair $(\iota,A)$ with $\iota \in \cI_n$ and $A\subseteq \Asc(\iota)$ of size $n-i$. Before we start the proof we need one fact about SYT. You can see this fact demonstrated in Figure~\ref{fig:BMproof}.

\bigskip

\noindent {\it Fact: }Given an involution $\iota$ and its associated SYT $P$ then $i\in \Asc(\iota)$ if and only if $i+1$ appears in a column to the left of $i$ in $P$ and either the same row or in a higher row.

\begin{proof}[Proof of Theorem~\ref{MainThm}]
Say we start with a pair $(\iota,A)$ with $\iota\in \cI_n$ and  $A\subseteq \Asc(\iota)$ of size $n-i$. We will map this pair to a single proper SSYT $T$ with $\sh(T)\vdash n$ and $\ell(C(T))=i$. The involution $\iota$ is associated to its permutation matrix which is symmetric and by  RSK  is further associated to a SYT $P$. We will construct a proper SSYT $T$ with $\sh(T)\vdash n$ and $\ell(C(T))=i$ from $P$ and $A$ as follows. For all $a\in A$ find the two cells in $P$ which are filled with $a$ and $a+1$. Change the cell filled with $a+1$ so that it matches the filling of the other cell. Call this  $P'$. Note that this means if $a,a+1,\dots ,b\in A$ then all the cells filled with $a,a+1,\dots, b+1$ will have the same filling in $P'$. Since $|A|=n-i$ we know that $P'$ is filled with $i$ different letters. Note that we only matched fillings for ascents which means $P'$ is weakly increasing along rows and is still strictly increasing along columns using the fact above. Let $T$ be the standardization of $P'$,  $T=\text{st}(P')$, which is a proper SSYT with $\ell(C(T))=i$. See Figure~\ref{fig:BMproof} for an example. This assures us that the bijection we are describing is  well defined. 

\begin{figure}[h]
\begin{center}
\begin{tikzpicture}[scale = .5]
\draw (0,0) node { \young(11124,335,5) };
\draw (3.5,0) node {$\longleftrightarrow$};
%\draw (18,-2) arc (-30:30:5) ;
%\draw (7,3) arc (150:210:5) ;
\draw (9,0) node {$\left( 125836947,\{1,2,5,8\}\right)$ };
\begin{scope}[shift={(18,0)}]
\draw (0,0) node {\young(12347,569,8)};
\draw (3.5,.5) node {$\overset{\text{RSK}}{\longleftrightarrow}$};
\draw (6.5,0) node {$ 125836947$};
\end{scope}
 \end{tikzpicture}
\end{center}
\caption{Above we have the pair $(\iota,A)$ which is associated to the SSYT $T$ on the left with  $\ell(C(T))=5$ and  $|A|=9-5$. }
\label{fig:BMproof}
\end{figure}
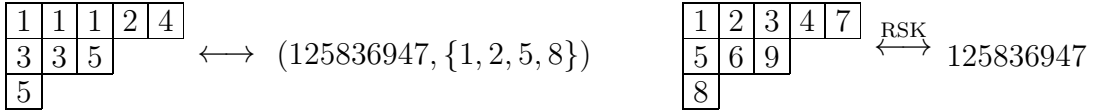

To show that this is actually a bijection we will show that the map we just described has an inverse. Consider a proper SSYT $T$ with $\sh(T)\vdash n$ and $\ell(C(T))=i$. We will construct a SYT $P$ by labeling all the $n$ cells in $T$ with $1,2,\dots, n$ and then replace each cell with the label. We will also keep track of some of these labels since they will form the set $A$. If $T$ has $m_1$ $1$s then we label these cells from left to right with $1,2,\dots, m_1$. This makes sense since no column will contain more than one $1$. Let $L_1=\{1,2,\dots, m_1-1\}$  be the set of all these labels but the largest one, $m_1$. If there are $m_2$ $2$s we will label the cells with $2$s left to right with  $m_1+1,m_1+2,\dots,m_1+m_2$. Let $L_2$ contain all these labels but the largest one, $m_1+m_2$. If we have $m_i$ $i$'s we label the cells with $i$ in $T$ left to right with the integers starting from $m_1+\cdots+m_{i-1}+1$ to $m_1+\cdots+m_{i}$. Let $L_i$ contain all these labels but the largest one. 
First, note that $P$ is a SYT because  no two $i$s were in the same column in $T$ which results in $P$ being strictly increasing along rows and columns. By RSK $P$ is associated to an involution $\iota$.
Second, note that by the above fact and the position of the cells containing $a\in L_i$ and $a+1$ in $P$ that $a\in \Asc(\iota)$ and $A=\cup_i L_i\subseteq \Asc(\iota)$.  Because $T$ was a proper SSYT with content length $i$ we can conclude that $|A|=n-i$. This completes the inverse map from $T$ to the pair $(\iota,A)$. 

Lastly, we only need to argue that these two maps are inverses. This is because in the SYT $P$ the subset $A$ of ascents tell of which cells in $T$ will have the matching fillings. In the SSYT $T$ the cells with matching fillings tell us which  ascents of $\iota$ should be in the set $A$. 
\end{proof}

\end{section}

%\nocite{*}
\bibliographystyle{alpha}
\bibliography{dahlref}{}

\end{document}